\documentclass[11pt]{amsart}

\usepackage{amsmath,amsthm,amssymb,mathrsfs,verbatim,graphicx,subfig}
\usepackage[tmargin=1.0in,bmargin=1.0in,rmargin=1.0in,lmargin=1.0in]{geometry}

\theoremstyle{plain}

\newtheorem{theorem}{\textrm{\textbf{Theorem}}}[section]
\newtheorem*{utheorem}{\textrm{\textbf{Theorem}}}

\newtheorem{corollary}[theorem]{\textrm{\textbf{Corollary}}}
\newtheorem{proposition}[theorem]{\textrm{\textbf{Proposition}}}
\newtheorem{lemma}[theorem]{\textrm{\textbf{Lemma}}}

\theoremstyle{definition}

\newtheorem{definition}[theorem]{\textrm{\textbf{Definition}}}

\newtheorem{remark}[theorem]{\textrm{\textbf{Remark}}}

\theoremstyle{remark}

\newtheorem*{acknowledgement}{Acknowledgment}

\numberwithin{equation}{section}

\def\sgn {\mathop{\rm sgn}}
\def\diag {\mathop{\rm diag}}
\def\cond {\mathop{\rm cond}}
\def\hol {\mathop{\rm Hol}}
\def\tr{\mathop{\rm trace}}

\def\mc#1{\multicolumn{1}{|c}{#1}}

\begin{document}
\title{Functions preserving positive definiteness for sparse matrices}
\author{Dominique Guillot \and Bala Rajaratnam \\ Stanford University}

\date{}
\maketitle
\begin{abstract}
We consider the problem of characterizing entrywise functions that preserve the cone of positive definite matrices when applied to every off-diagonal element. Our results extend theorems of Schoenberg [Duke Math. J. 9], Rudin [Duke Math. J. 26], Christensen and Ressel [Trans. Amer. Math. Soc., 243], and others,  where similar problems were studied when the function is applied to all elements, including the diagonal ones. It is shown that functions that are guaranteed to preserve positive definiteness cannot at the same time induce sparsity, i.e., set elements to zero. These results have important implications for the regularization of positive definite matrices, where functions are often applied to only the off-diagonal elements to obtain sparse matrices with better properties (e.g., Markov random field/graphical model structure, better condition number). As a particular case, it is shown that \emph{soft-thresholding}, a commonly used operation in modern high-dimensional probability and statistics, is not guaranteed to maintain positive definiteness, even if the original matrix is sparse. This result has a deep connection to graphs, and in particular, to the class of trees. We then proceed to fully characterize functions which do preserve positive definiteness. This characterization is in terms of absolutely monotonic functions and turns out to be quite different from the case when the function is also applied to diagonal elements. We conclude by giving bounds on the condition number of a matrix which guarantee that the regularized matrix is positive definite. 
\end{abstract}

\section{Introduction}\label{sec:intro}
In one of his celebrated papers, \emph{Positive definite functions on spheres} \cite{Schoenberg42}, I.J. Schoenberg proved that every continuous function $f: (-1,1) \rightarrow \mathbb{R}$ having the property that the matrix $(f(a_{ij}))$ is positive semidefinite for every symmetric positive semidefinite matrix $(a_{ij})$ with entries in $(-1,1)$ has a power series representation with nonnegative coefficients. Functions satisfying this latter property are often known as \emph{absolutely monotonic functions}. The aforementioned result has been generalized by Rudin \cite{Rudin59} who showed that the class of absolutely monotonic functions fully characterizes the class of (not necessarily continuous) functions mapping every positive (semi)definite sequence to a positive (semi)definite sequence. Equivalently, the class of absolutely monotonic functions are exactly the functions mapping sequences of Fourier--Stieltjes coefficients to sequences of Fourier--Stieltjes coefficients. 

In this paper, we revisit and extend Schoenberg's results with important modern applications in mind. Positive definite matrices arise naturally as covariance or correlation matrices. Consider an $n \times n$ covariance (or correlation) matrix $\Sigma$. In modern high-dimensional probability and statistics, two of the most common techniques employed to improve the properties of $\Sigma$ are the so-called \emph{hard-thresholding} and \emph{soft-thresholding} procedures. Hard-thresholding a positive definite matrix entails setting small off-diagonal elements of $\Sigma$ to zero. This technique has the advantage of eliminating spurious or insignificant correlations, and leads to sparse estimates of the matrix $\Sigma$. These thresholded matrices generally have better properties (such as better conditioning, graphical model structure) and lead to models that are easier to store, interpret, and work with. At the same time, in contrast with most ``regularization'' techniques, this procedure incurs very little computational cost. Hence it can be applied to ultra high-dimensional matrices, as required by many modern-day applications (see \cite{Zhang_Horvath, Li_Horvath, bickel_levina, hero_rajaratnam, Guillot_Rajaratnam2012, Hero_Rajaratnam2012}). 

An important property of thresholded covariance matrices that is generally required for applications is positive definiteness. Nonetheless, regularization procedures such as hard-thresholding are often used indiscriminately, and with very little attention paid to the algebraic properties of the resulting thresholded matrices. It is therefore critical to understand whether or not the cone of positive definite matrices is invariant with respect to hard-thresholding (and other similar operations), especially in order for these regularization methods to be widely applicable. 

We now formalize some notation. Given $\epsilon > 0$, the hard-thresholding operation is equivalent to applying the function $f^H_\epsilon : \mathbb{R} \rightarrow \mathbb{R}$ defined by 
\begin{equation}\label{eqn:hard_thres}
  f^H_\epsilon(x) = \left\{\begin{array}{ll}
  x & \textrm{ if } |x| > \epsilon \\
  0 & \textrm{ otherwise}
  \end{array}
  \right.
\end{equation}
to every off-diagonal element of the matrix $\Sigma$. As mentioned above, modern probability and statistics require that the thresholding function is applied only to off-diagonal elements. As a consequence, previous results from the mathematics literature cannot be directly used to determine whether hard-thresholding and other similar techniques preserve positive definiteness. The aim of this paper is to investigate this important question, especially given its significance in contemporary mathematical sciences.  

Algebraic properties of hard-thresholded matrices have been studied in detail in \cite{Guillot_Rajaratnam2012}, where it is shown that, even if the original matrix is sparse, hard-thresholding is not guaranteed to preserve positive definiteness. Thus the function $f_\epsilon^H$ does not map the cone of positive definite matrices into itself.  

A type of function that is equally frequently used in the literature is the so-called \emph{soft-thresholding} function $f^S_\epsilon : \mathbb{R} \rightarrow \mathbb{R}$, given by
\begin{equation}\label{eqn:soft_thres}
f^S_\epsilon(x) = \sgn(x) (|x| - \epsilon)_+, 
\end{equation}
where $\sgn(x)$ denotes the sign of $x$ and $(a)_+ = \max(a,0)$. Compared to hard-thresholding, soft-thresholding continuously shrinks all elements of a matrix to zero, thus giving more hope of preserving positive definiteness than hard-thresholding. To the authors' knowledge, a detailed analysis of whether or not this is true has not been undertaken in the literature. It is also natural to ask whether the hard or soft-thresholding function can be replaced by other functions in order to induce sparsity (i.e., zeros) in positive definite matrices and, at the same time, maintain positive definiteness.

The first theorem of this paper extends results from \cite{Guillot_Rajaratnam2012} and shows the rather surprising result that, for a given positive definite matrix, even if it is already sparse, there is generally no guarantee that its soft-thresholded version will remain positive definite. We state this result below: 

\begin{utheorem}
Let $G = (V,E)$ be a connected undirected graph and denote by $\mathbb{P}_G^+$ the cone of symmetric positive definite matrices with zeros according to $G$ 
\begin{equation}
\mathbb{P}_G^+ := \{A = (a_{ij}) \in \mathbb{P}^+ : a_{ij} = 0 \textrm{ if } (i,j) \not\in E, i \not= j\}, 
\end{equation}
where $\mathbb{P}^+$ denotes the cone of all symmetric positive definite matrices. For $\epsilon > 0$, denote by $\eta_\epsilon(A)$ the soft-thresholded matrix 
\begin{equation}\label{eqn:soft_thres_matrix}
(\eta_\epsilon(A))_{ij} = \left\{\begin{array}{cc}\sgn(a_{ij}) (|a_{ij}|-\epsilon)_+ & \textrm{ if } i \not= j \\
a_{ij} & \textrm{ otherwise}\end{array}\right..
\end{equation}
Then the following are equivalent: 
\begin{enumerate}
\item There exists $\epsilon > 0$ such that for every $A \in \mathbb{P}_G^+$, we have $\eta_\epsilon(A) > 0$; 
\item For every $\epsilon > 0$ and every $A \in \mathbb{P}_G^+$, we have $\eta_\epsilon(A) > 0$;
\item $G$ is a tree. 
\end{enumerate}  
\end{utheorem}
Note that for a given matrix $A \in \mathbb{P}_G^+$, by the continuity of the eigenvalues, there exists $\epsilon > 0$ such that $\eta_\epsilon(A) > 0$. However, different matrices can lose positive definiteness for different values of $\epsilon$. The existence of a ``universal'' value $\epsilon_0 > 0$ with the property that $\eta_{\epsilon_0}(A) > 0$ for every $A \in \mathbb{P}_G^+$ would have tremendous practical implications. Indeed, if such an $\epsilon_0$ existed, matrices could be safely soft-thresholded to remove some of their small entries while retaining positive definiteness. The previous theorem asserts that, except when the structure of zeros of $A$ corresponds to a tree, such an $\epsilon_0$ unfortunately does not exist. 

Following the previous result, we extend Schoenberg's results by fully characterizing the functions that preserve positive definiteness when applied to every off-diagonal element. The statement of the main theorem of the paper is given below. 
\begin{utheorem}
Let $0 < \alpha \leq \infty$ and let $f: (-\alpha, \alpha) \rightarrow \mathbb{R}$. For every matrix $A = (a_{ij})$, denote by $f^*[A]$ the matrix
\begin{equation}
(f^*[A])_{ij} = \left\{\begin{array}{ll}f(a_{ij}) & \textrm{ if } i \not=j\\
a_{ij} & \textrm{ if } i=j \end{array}\right..
\end{equation}
Then $f^*[A]$ is positive semidefinite for every symmetric positive semidefinite matrix $A$ with entries in $(-\alpha, \alpha)$ if and only if $f(x) = x g(x)$ where: 
\begin{enumerate}
\item $g$ is analytic on the disc $D(0,\alpha)$; 
\item $\|g\|_\infty \leq 1$;
\item $g$ is absolutely monotonic on $(0, \alpha)$. 
\end{enumerate}
When $\alpha = \infty$, the only functions satisfying the above conditions are the affine functions $f(x) = ax$ for $0 \leq a \leq 1$.  
\end{utheorem}
The above result does come as a surprise. It formally demonstrates that, except in trivial cases, no guarantee can be given that applying a function to the off-diagonal elements of a matrix will preserve positive definiteness. There are thus no theoretical safeguards that thresholding procedures used in innumerable applications will maintain positive definiteness. 

The remainder of the paper is structured as follows. Section \ref{sec:rev:hard} reviews results that have been recently established for hard-thresholding. In Section \ref{sec:soft_thres}, a characterization of matrices preserving positive definiteness upon soft-thresholding is given. The characterization turns out to have a non-trivial relationship to graphs and the structure of zeros in the original matrix. Section \ref{sec:gen_thres} then studies the behavior of positive semidefinite matrices when an arbitrary function $f$ is applied to every element of the matrix. A review of previous results from the literature is first given. The results are then extended to include the case where the function is applied only to the off-diagonal elements of the matrix. A complete characterization of functions preserving positive definiteness in this modern setting is given. Finally, Section \ref{sec:eig_inequalities} gives sufficient conditions for a matrix $A$ and a function $f$ so that the matrix $f^*[A]$ remains positive definite. In particular, it is shown that the matrix $f^*[A]$ is guaranteed to be positive definite as long as the condition number of $A$ is smaller than an explicit bound.  
\ \\ \ \\
\emph{Notation: }
Throughout the paper, we shall make use of the following graph theoretic notation. Let $G=(V,E)$ be an undirected graph with $n \geq 1$ vertices $V = \{1, \dots, n\}$ and edge set $E$. Two vertices $a, b \in V$, $a \not= b$, are said to be \emph{adjacent} in $G$ if $(a,b) \in E$. A graph is \emph{simple} if it is undirected, and does not have multiple edges or self-loops. We will only work with finite simple graphs in this paper. 

We say that the graph $G'=(V', E')$ is a subgraph of  $G=(V, E)$,  denoted by  $G'\subset G$,  if $V'\subseteq V$ and  $E'\subset E$. In addition, if  $G'\subset G$  and  $E'=(V'\times V')\cap E$, we say that $G'$ is an \emph{induced subgraph} of  $G$. A graph  $G$  is called \emph{complete} if every pair of vertices are adjacent. A \emph{path} of length  $k\geq 1$  from vertex $i$  to  $j$  is a finite sequence of distinct vertices  $v_0=i,\ldots, v_k=j$  in $V$  and edges $(v_0,v_1), \ldots, (v_{k-1}, v_k)\in E$. A \emph{$k$-{cycle}} in  $G$  is a path of length  $k-1$ with an additional edge connecting the two end points. A graph $G$ is called \emph{connected} if for any pair of distinct vertices $i, j\in V$ there exists a path between them. 

A special class of graphs are \emph{trees}. These are connected graphs on $n$ vertices with exactly $n-1$ edges. A tree can also be defined as a connected graph with no cycle of length $n \geq 3$, or as a connected graph with a unique path between any two vertices. 

Graphs provide a useful way to encode patterns of zeros in symmetric matrices by letting $(i,j) \in E$ if and only if $a_{ij} \not= 0$. Denote by $\mathbb{P}_n^+$ the cone of $n \times n$ symmetric positive definite matrices,  and by $\mathbb{P}^+$ the cone of symmetric positive definite matrices (of any dimension). We shall write $A > 0$ whenever $A \in \mathbb{P}^+$ and $A > B$ if $A-B \in \mathbb{P}^+$. Similarly, we write $A \geq 0$ whenever $A$ is symmetric positive semidefinite, and $A \geq B$ if $A-B \geq 0$. We define the cone of symmetric positive definite matrices with zeros according to a given graph $G$ with $n$ vertices by 
\begin{equation}
\mathbb{P}_G^+ := \{A \in \mathbb{P}_n^+ : a_{ij} = 0 \textrm{ if } (i,j) \not\in E, i \not=j\}.
\end{equation}
Denoting the space of $n \times n$ matrices by $\mathbb{M}_n$, recall that a $(n_1+n_2) \times (n_1 + n_2)$ symmetric block matrix 
\[
M = \left(\begin{array}{cc}A & B \\ B^t & D\end{array}\right)
\]
where $A \in \mathbb{M}_{n_1 \times n_1}, B \in \mathbb{M}_{n_1 \times n_2}$, and $D \in \mathbb{M}_{n_2 \times n_2}$, is positive definite if and only if $D$ is positive definite and $S_1 = A - BD^{-1}B^t$ is positive definite. The matrix $S_1$ is called the \emph{Schur complement} of $D$ in $M$. Alternatively, $M$ is positive definite if and only if $A$ is positive definite and $S_2 = D-B^tA^{-1}B$ is positive definite. The matrix $S_2$ is called the \emph{Schur complement} of $A$ in $M$. Finally, for a symmetric matrix $A$, we shall denote by $\lambda_\textrm{min}(A)$ and $\lambda_\textrm{max}(A)$ its smallest and largest eigenvalues respectively. 

%%%%%
\section{Review of relevant results on hard-thresholding}\label{sec:rev:hard}

Algebraic properties of hard-thresholding have been studied in \cite{Guillot_Rajaratnam2012}. In particular, two types of hard-thresholding operations have been considered. Let $G$ be a graph with $n$ vertices. The graph $G$ induces a \emph{hard-thresholding} operation, mapping every symmetric $n \times n$ matrix $A = (a_{ij})$ to a matrix $A_G$ defined by 
\begin{equation}
(A_G)_{ij} = \left\{\begin{array}{ll} a_{ij} & \textrm{if } (i,j) \in E \textrm{ or } i=j \\ 0 & \textrm{otherwise}.\end{array}\right.
\end{equation}
We say that the matrix $A_G$ is obtained from $A$ by \emph{thresholding $A$ with respect to the graph $G$}. 

The following result from \cite{Guillot_Rajaratnam2012} fully characterizes the graphs preserving positive definiteness upon thresholding. 

\begin{theorem}[{\cite[Theorem 3.1]{Guillot_Rajaratnam2012}}]\label{thm:ht_complete}
Let $A$ be an arbitrary symmetric $n \times n$ matrix such that $A > 0$, i.e., $A \in \mathbb{P}_n^+$. Threshold $A$ with respect to a graph $G = (V,E)$ with the resulting thresholded matrix denoted by $A_G$. Then
\begin{equation}
A_G > 0 \textrm{ for all } A \in \mathbb{P}^+ \Leftrightarrow G = \bigcup_{i=1}^\tau G_i \qquad \textrm{ for some } \tau \in \mathbb{N},  
\end{equation}
where $G_i$, $i=1, \dots, \tau$, denote disconnected, complete components of $G$. 
\end{theorem}

The above theorem asserts that a positive definite matrix $A$ is guaranteed to retain positive definiteness upon thresholding with respect to a graph $G$ only in the trivial case when the thresholded matrix can be reorganized as a block diagonal matrix where, within each block, there is no thresholding. This result can be further generalized to matrices in $\mathbb{P}_G^+$ which are thresholded with respect to a subgraph $H$ of $G$. The following theorem shows that thresholding matrices from this class yields essentially the same results as in the complete graph case. 

\begin{theorem}[{\cite[Theorem 3.3]{Guillot_Rajaratnam2012}}]\label{Th_Hard_Thres_General}
Let $G = (V,E)$ be an undirected graph and let $H = (V,E')$ be a subgraph of $G$ i.e., $E' \subset E$. Then $A_H > 0$ for every $A \in \mathbb{P}_G^+$ if and only if $H = G_1 \cup \dots \cup G_k$ where $G_1, \dots, G_k$ are disconnected induced subgraphs of $G$. 
\end{theorem}

Theorems \ref{thm:ht_complete} and \ref{Th_Hard_Thres_General} treat the case of thresholding elements regardless of their magnitude. In practical applications however, in order to induce sparsity, hard-thresholding is often performed on the smaller elements of the positive definite matrix. The following result shows that only matrices with zeros according to a tree are guaranteed to retain positive definiteness when hard-thresholded at a given level $\epsilon > 0$.  

\begin{definition}
The matrix $B$ is said to be the \emph{hard-thresholded version of $A$ at level $\epsilon$} if $b_{ij} = a_{ij}$ when $|a_{ij}| > \epsilon$ or $i=j$, and $b_{ij}=0$ otherwise. 
\end{definition}

\begin{theorem}[{\cite[Theorem 3.6]{Guillot_Rajaratnam2012}}]\label{Th_hard_thresholding_level_ep}
Let $G$ be a connected undirected graph. The following are equivalent: 
\begin{enumerate}
\item There exists $\epsilon > 0$ such that for every $A \in \mathbb{P}_G^+$, the hard-thresholded version of $A$ at level $\epsilon$ is positive definite; 
\item For every $\epsilon > 0$ and every $A \in \mathbb{P}_G^+$, the hard-thresholded version of $A$ at level $\epsilon$ is positive definite; 
\item $G$ is a tree. 
\end{enumerate}
\end{theorem}

The result above demonstrates that hard-thresholding positive definite matrices at a given level $\epsilon$ can also quickly lead to a loss of positive definiteness, though it is not as severe as when thresholding with respect to a graph. Recall that hard-thresholding a matrix $A$ at level $\epsilon$ is equivalent to applying the hard-thresholding function given in \eqref{eqn:hard_thres} to every off-diagonal element of $A$. It is thus natural to replace the hard-thresholding function by other functions to see if positive definiteness can be retained. A popular alternative is the soft-thresholding function (see \eqref{eqn:soft_thres}, \eqref{eqn:soft_thres_matrix}, and Figure \ref{fig:hard_soft}). The next section is devoted to studying the algebraic properties of soft-thresholded positive definite matrices. We conclude this section by noting that Theorem \ref{Th_hard_thresholding_level_ep} also yields a characterization of trees via thresholding  matrices. 

\begin{figure}
\centering
  \subfloat[Hard-thresholding]{
	\begin{minipage}[c][0.8\width]{
	   0.4\textwidth}
	   \centering
	   \includegraphics[width=6.5cm]{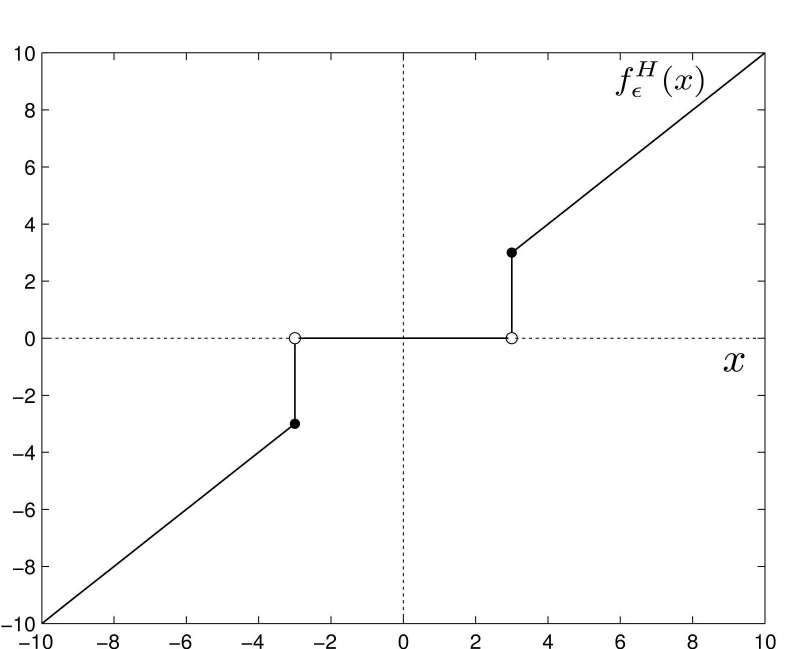}
	   \label{fig:hard_thres}
	\end{minipage}} 	
  \subfloat[Soft-thresholding]{
	\begin{minipage}[c][0.8\width]{
	   0.4\textwidth}
	   \centering
	   \includegraphics[width=6.5cm]{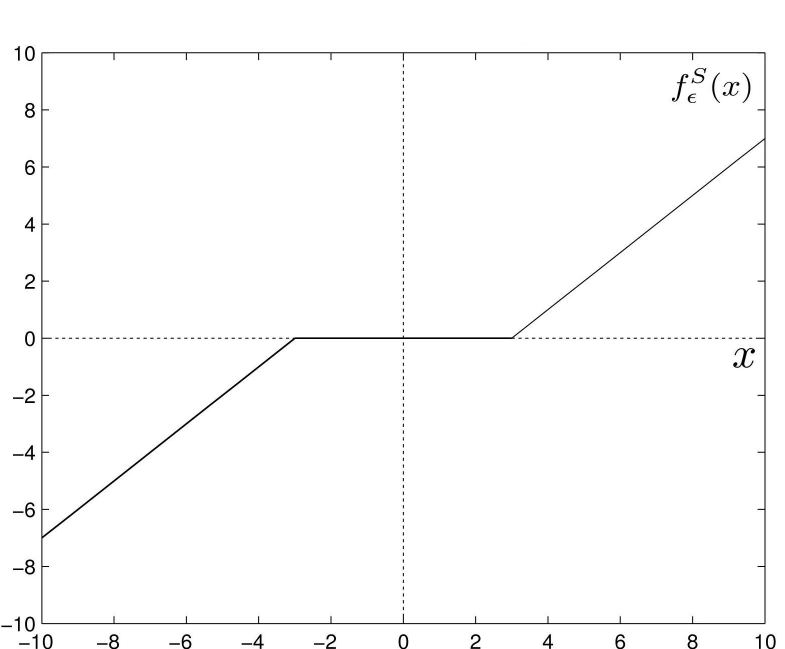}
	   \label{fig:soft_thres}
	\end{minipage}}
\caption{Illustration of the hard and soft-thresholding functions with $\epsilon = 3$}
\label{fig:hard_soft}
\end{figure}

%%%%%
\section{Soft-thresholding}\label{sec:soft_thres}
We now proceed to the more intricate task of characterizing the graphs $G$ for which every matrix $A \in \mathbb{P}_G^+$ retains positive definiteness when soft-thresholded at a given level $\epsilon > 0$. As soft-thresholding is a continuous function as opposed to the hard-thresholding function, it would seem that soft-thresholding may have better properties in terms of retaining positive definiteness. 

\begin{definition}
For a matrix $A = (a_{ij})$ and $\epsilon > 0$, the \emph{soft-thresholded} version of $A$ at level $\epsilon$ is given by:  
\begin{equation}
(\eta_\epsilon(A))_{ij} = \left\{\begin{array}{cc}\sgn(a_{ij}) (|a_{ij}|-\epsilon)_+ & \textrm{ if } i \not= j \\
a_{ij} & \textrm{ otherwise}\end{array}\right..
\end{equation}
\end{definition}

\begin{theorem}\label{Th_Princ_Soft_Thres}
Let $G = (V,E)$ be a connected undirected graph. Then the following are equivalent: 
\begin{enumerate}
\item There exists $\epsilon > 0$ such that for every $A \in \mathbb{P}_G^+$, we have $\eta_\epsilon(A) > 0$; 
\item For every $\epsilon > 0$ and every $A \in \mathbb{P}_G^+$, we have $\eta_\epsilon(A) > 0$;
\item $G$ is a tree. 
\end{enumerate}  
\end{theorem}

\begin{remark}
Regardless of the continuity of the soft-thresholding function $f_\epsilon^S$, Theorem \ref{Th_Princ_Soft_Thres} demonstrates that soft-thresholding has the same effect as hard-thresholding when it comes to retaining positive definiteness (see Theorem \ref{Th_hard_thresholding_level_ep}). Theorem \ref{Th_Princ_Soft_Thres} also gives yet another characterization of trees.
\end{remark}

\begin{remark}
The proof of Theorem 3.2 given below for soft-thresholding is more challenging as compared to the proof of Theorem 2.4 for hard-thresholding. In [2], an explicit example of a matrix $A \in \mathbb{P}^+_{C_n}$ losing positive definiteness upon hard-thresholding is constructed for all $n \geq 3$. A direct construction of a matrix losing positive definiteness when soft-thresholded is elusive. The proof below proceeds by induction:  we start with a matrix $A_3 \in \mathbb{P}^+_{C_3}$ losing positive definiteness when soft-thresholded at level $\epsilon = 0.1$. First, the matrix $A_3$ is determined numerically. Thereafter, a matrix $A_n \in \mathbb{P}^+_{C_n}$ losing positive definiteness when soft-thresholded at the same level is then constructed inductively by exploiting properties of Schur complements. 
\end{remark}

\begin{proof}[Proof of Theorem \ref{Th_Princ_Soft_Thres}]
($1 \Rightarrow 3$) We shall prove the contrapositive form. Let $C_n$ denote the cycle graph with $n$ vertices. Recall that a tree is a graph without cycle of length $n \geq 3$. Thus, if $G$ is not a tree, then it contains a cycle of length greater or equal than $3$. Therefore, to prove this part of the result, it is sufficient to construct, for every $n \geq 3$, a positive definite matrix $A_n \in \mathbb{P}_{C_n}^+$ which does not retain positive definiteness when soft-thresholded at the given level $\epsilon > 0$. We will begin by providing such examples of matrices for a fixed value of $\epsilon = 0.1$. We will then show how matrices with the same properties can be built for arbitrary values of $\epsilon > 0$. 

The following matrix  
\begin{equation}
A_3 := \begin{pmatrix}
8.3 & 1.1 & 1.1 \\
1.1 & 0.3 & 0.1 \\
1.1 & 0.1 & 0.2
\end{pmatrix}
\end{equation}
provides an example for $n=3$, with threshold level $\epsilon=0.1$. Also, notice that 1) the matrix $\widetilde{A}_3$ where 
\begin{equation}
\widetilde{A}_3 := \begin{pmatrix}
8.3 & 1.1 & 0 \\
1.1 & 0.3 & 0.1 \\
0 & 0.1 & 0.2
\end{pmatrix}, 
\end{equation}
which is $A_3$ with the $(1,3)$ and $(3,1)$ elements set to zero, is positive definite, and 2) the matrix $A$ stays positive definite when only the $(1,3)$ and $(3,1)$ elements are soft-thresholded at level $\epsilon = 0.1$. We will construct a similar matrix $A_n$ for $n \geq 4$ inductively. Properties 1) and 2) will be important to perform the induction step. 

Indeed, assume that, for some $n \geq 3$, there exists a matrix $A_n \in \mathbb{P}_{C_n}^+$ which loses positive definiteness when soft-thresholded at level $\epsilon = 0.1$. Let us assume also that the matrix $\widetilde{A}_n$ obtained from $A_n$ by setting the $(1,n)$ and $(n,1)$ elements to $0$ is positive definite and that the matrix obtained from $A_n$ by soft-thresholding only the $(1,n)$ and $(n,1)$ elements at level $\epsilon$ is positive definite. These properties are satisfied for $n=3$ by the matrix $A_3$ given above. We will build a matrix $A_{n+1} \in \mathbb{P}_{C_{n+1}}^+$ satisfying the same properties. Let $a_n$ denote the $(1,n)$ element of $A_n$. For every real number $r$, let $r_\epsilon := \sgn(r)(|r|-\epsilon)_+$ denote the value of $r$ soft-thresholded at level $\epsilon$. To simplify the notation, let us denote by $a_{n, \epsilon}$ the value of $(a_n)_\epsilon$. Now consider the matrix 
\begin{equation}
A_{n+1} := \left(\begin{array}{cccccc}
 & & & & & \mc{a_{n+1}} \\
 & & & & & \mc{0} \\
 & & \widetilde{A}_{n} + D_n & & & \mc{\vdots} \\
 & & & & & \mc{0} \\
 & & & & & \mc{b} \\
\cline{1-6}
a_{n+1} &0 & \dots & 0& b & \mc{\alpha}
\end{array}\right)
\end{equation}
where $D_n$ is a diagonal matrix with positive diagonal $((a_{n+1,\epsilon})^2/\alpha, 0,\dots,0, (b_\epsilon)^2/\alpha)$. Notice that $A_{n+1}$ has zeros according to ${C_{n+1}}$. We will prove that $a_{n+1}, b, \alpha$ can be chosen so that $A_{n+1}$ satisfies the required properties. 

Let us first choose the value of $a_{n+1}$ as a function of $\alpha$ and $b$ in such a way that
\begin{equation}\label{eqn:defn_a} 
\frac{-a_{n+1,\epsilon}\ b_\epsilon}{\alpha} = a_{n,\epsilon}.
\end{equation} 
This is always possible if $|b| > \epsilon$. Indeed, if $|b| > \epsilon$, then $a_{n+1}$ satisfies equation (\ref{eqn:defn_a}) for 
\begin{equation}
a_{n+1} := -\alpha \frac{a_{n, \epsilon} }{b_{\epsilon}} + s \epsilon
\end{equation}
where $s = \sgn(-\alpha a_{n, \epsilon}/b_{\epsilon})$. 

We claim that we can choose $\alpha > 0$ and $b > \epsilon $ such that: 
\begin{enumerate}
\item $A_{n+1}$ is positive definite; 
\item $\widetilde{A}_{n+1}$ is positive definite; 
\item $A_{n+1}$ is not positive definite when soft-thresholded at level $\epsilon$, i.e., $\eta_\epsilon(A_{n+1}) \not> 0$;
\item $A_{n+1}$ is positive definite when only its $(1,n+1)$ and $(n+1,1)$ elements are soft-thresholded at level $\epsilon$. 
\end{enumerate}
Conditions (1) and (3) are the two conditions needed to prove that the matrix $A_{n+1}$ satisfies the theorem. Conditions (2) and (4) are required in the induction step. 

First, note that the matrix $A_{n+1}$ has been constructed in such a way that the Schur complement of $\alpha$ in $\eta_\epsilon(A_{n+1})$ is equal to $\eta_\epsilon(A_n)$. Therefore, by the induction hypothesis, $\eta_\epsilon(A_{n+1})$ is not positive definite for any value of $|b| > \epsilon$ and $\alpha > 0$. This proves (3). 

Since $\alpha > 0$, to prove properties (1), (2) and (4), we only need to study the Schur complement of $\alpha $ in the three matrices: $A_{n+1}$, $\widetilde{A}_{n+1}$ and in the matrix obtained from $A_{n+1}$ by soft-thresholding the $(1,n+1)$ and $(n+1,1)$ elements. We will prove that properties (1), (2) and (4) hold true asymptotically as $\alpha, b \rightarrow \infty$. Therefore, the result will follow by choosing appropriately large values of $\alpha$ and $b$.

The Schur complement of $\alpha$ in  $A_{n+1}$ is given by
\begin{equation}\label{eqn:schur_compl}
\widetilde{A}_n + D_n - \left(\begin{array}{ccc}
\frac{a_{n+1}^2}{\alpha} & \dots & \frac{a_{n+1} b}{\alpha} \\
\vdots & \ddots & \vdots \\
\frac{a_{n+1} b}{\alpha} & \dots & \frac{b^2}{\alpha} 
\end{array}\right) = \widetilde{A}_n  + \left(\begin{array}{ccc}
\frac{(a_{n+1,\epsilon})^2-a_{n+1}^2}{\alpha} & \dots & -\frac{a_{n+1} b}{\alpha} \\
\vdots & \ddots & \vdots \\
-\frac{a_{n+1} b}{\alpha} & \dots & \frac{b_\epsilon^2-b^2}{\alpha} 
\end{array}\right)
\end{equation}
where the dots in the above matrices represent zeros. Let us take $\alpha = b^3$. Since $a_{n+1}$ and $\alpha$ depend on the value of $b$ and since $\epsilon$ is fixed, $b$ becomes the only ``free'' parameter. We will prove that properties (1), (2) and (4) hold for large values of $b$. We begin by studying the limiting behavior of different quantities related to the Schur complement (\ref{eqn:schur_compl}). We will show that 
\begin{align}
\frac{a_{n+1,\epsilon}}{\alpha} &\rightarrow 0 \textrm{ as } b \rightarrow \infty \label{eqn:limit1}\\ 
\frac{a_{n+1}^2}{\alpha} - \frac{(a_{n+1, \epsilon})^2}{\alpha} &\rightarrow 0 \textrm{ as } b \rightarrow \infty \label{eqn:limit2}\\
\frac{a_{n+1} b}{\alpha} &\rightarrow -a_{n, \epsilon}  \textrm{ as } b \rightarrow \infty \label{eqn:limit3}.  
\end{align}
Equation (\ref{eqn:limit1}), follows easily by equation (\ref{eqn:defn_a}), since  
\begin{equation}
a_{n+1,\epsilon}/\alpha = -a_{n,\epsilon}/b_\epsilon \rightarrow 0 \textrm{ as } b \rightarrow \infty. 
\end{equation}
Now to prove \eqref{eqn:limit2}, recall that, by construction, $a_{n+1} = a_{n+1,\epsilon} \pm \epsilon$ where the sign depends on the sign of $a_{n+1}$. Therefore 
\begin{equation}
\frac{(a_{n+1})^2}{\alpha} - \frac{(a_{n+1, \epsilon})^2}{\alpha} = \frac{(a_{n+1, \epsilon} \pm \epsilon)^2}{\alpha} - \frac{(a_{n+1, \epsilon})^2}{\alpha} = \frac{\pm 2\epsilon a_{n+1, \epsilon}}{\alpha} + \frac{\epsilon^2}{\alpha}. 
\end{equation}
The first term tends to $0$ as $b \rightarrow \infty$ as shown above. Also, since $\alpha = b^3$, $\alpha \rightarrow \infty$ as $b \rightarrow \infty$ and so $\epsilon^2/\alpha \rightarrow 0$ as $b \rightarrow \infty$. This proves equation (\ref{eqn:limit2}). 

Finally, for (\ref{eqn:limit3}), if $b > \epsilon$ then $b_\epsilon = b - \epsilon$ and 
\begin{eqnarray*}
\left|\frac{a_{n+1} b}{\alpha} - \frac{a_{n+1, \epsilon} b_\epsilon}{\alpha}\right| &=& \left|\frac{(a_{n+1, \epsilon} \pm \epsilon)(b_\epsilon + \epsilon)}{\alpha} - \frac{a_{n+1, \epsilon} b_\epsilon}{\alpha}\right| \\
&=& \left|\epsilon \frac{a_{n+1, \epsilon}}{\alpha} \pm \epsilon \frac{b_\epsilon}{\alpha} \pm \frac{\epsilon^2}{\alpha}\right|.
\end{eqnarray*}

As we have seen above, $a_{n+1, \epsilon}/\alpha \rightarrow 0$ as $b \rightarrow \infty$. Also, $b_\epsilon/\alpha \rightarrow 0$ as $b \rightarrow \infty$ since $\alpha = b^3$. Therefore, 
\begin{equation}
\frac{a_{n+1} b}{\alpha} - \frac{a_{n+1, \epsilon} b_\epsilon}{\alpha} \rightarrow 0
\end{equation}
as $b \rightarrow \infty$. But by (\ref{eqn:defn_a}), we have $a_{n+1, \epsilon} b_\epsilon/\alpha = -a_{n, \epsilon}$. Therefore
\begin{equation}
\frac{a_{n+1} b}{\alpha} \rightarrow -a_{n, \epsilon}
\end{equation}
as $b \rightarrow \infty$.

Using the results in equations (\ref{eqn:limit1})--(\ref{eqn:limit3}), we now proceed to show that properties (1), (2) and (4) hold true for appropriately large values of $b$. To prove (1), we only need to show that the Schur complement given by $(\ref{eqn:schur_compl})$ is positive definite for large values of $b$. Indeed, notice that from (\ref{eqn:limit2}) and (\ref{eqn:limit3}),  we have 
\begin{equation}
\left(\begin{array}{ccc}
\frac{(a_{n+1,\epsilon})^2-a_{n+1}^2}{\alpha} & \dots & -\frac{a_{n+1} b}{\alpha} \\
\vdots & \ddots & \vdots \\
-\frac{a_{n+1} b}{\alpha} & \dots & \frac{b_\epsilon^2-b^2}{\alpha} 
\end{array}\right) \rightarrow \left(\begin{array}{ccc}
0 & \dots & a_{n, \epsilon} \\
\vdots & \ddots & \vdots \\
a_{n, \epsilon} & \dots & 0
\end{array}\right)
\end{equation}
elementwise. Therefore the Schur complement of $\alpha$ in $A_{n+1}$ given in (\ref{eqn:schur_compl}) tends to 
\begin{equation}
\widetilde{A}_n + \left(\begin{array}{ccc}
0 & \dots & a_{n, \epsilon} \\
\vdots & \ddots & \vdots \\
a_{n, \epsilon} & \dots & 0
\end{array}\right). 
\end{equation}
as $b \rightarrow \infty$. This matrix is exactly the matrix $A_n$ with the $(1,n)$ and $(n,1)$ elements soft-thresholded at level $\epsilon$. Therefore, by the induction hypothesis, this matrix is positive definite and so is $A_{n+1}$ for large values of $b$. This proves property (1). 

To prove property $(2)$ note that the Schur complement of $\alpha$ in $\widetilde{A}_{n+1}$ is given by 
\begin{equation}
\widetilde{A}_n  + \left(\begin{array}{ccc}
\frac{(a_{n+1,\epsilon})^2}{\alpha} & \dots & 0 \\
\vdots & \ddots & \vdots \\
0& \dots & \frac{b_\epsilon^2-b^2}{\alpha} 
\end{array}\right).
\end{equation}
Notice that the $(1,1)$ entry of the righthand term is always positive whereas the $(n,n)$ element tends to $0$ as $b \rightarrow \infty$. Since the matrix $\widetilde{A}_n$ is positive definite by the induction hypothesis, the Schur complement of $\alpha$ in $\widetilde{A}_{n+1}$ is therefore also positive definite when $b$ is sufficiently large. This proves (2). 

Similarly, to prove (4), let us consider the Schur complement of $\alpha$ in the matrix $A_{n+1}$ with the $(1,n+1)$ and $(n+1,1)$ entries soft-thresholded at level $\epsilon$
\begin{equation}
\widetilde{A}_n  + \left(\begin{array}{ccc}
\frac{a_{n+1}^2 - (a_{n+1, \epsilon})^2}{\alpha} & \dots & -\frac{a_{n+1, \epsilon} b}{\alpha} \\
\vdots & \ddots & \vdots \\
-\frac{a_{n+1, \epsilon} b}{\alpha}& \dots & \frac{b_\epsilon^2-b^2}{\alpha} 
\end{array}\right).
\end{equation}
We have 
\begin{equation}
\frac{a_{n+1, \epsilon} b}{\alpha} = \frac{a_{n+1, \epsilon} (b_\epsilon + \epsilon)}{\alpha}  = \frac{a_{n+1, \epsilon} b_\epsilon}{\alpha}  + \epsilon \frac{a_{n+1, \epsilon}}{\alpha}. 
\end{equation}  
From (\ref{eqn:defn_a}) and (\ref{eqn:limit1}), we therefore have 
\begin{equation}
\frac{a_{n+1, \epsilon} b}{\alpha} \rightarrow -a_{n, \epsilon}
\end{equation}
as $b \rightarrow \infty$ and so the preceding Schur complement is asymptotic to the matrix $A_n$ with the $(1,n)$ and $(n,1)$ elements soft-thresholded at level $\epsilon$. By the induction hypothesis, this matrix is positive definite and therefore the same is true for the matrix $A_{n+1}$  with the $(1,n+1)$ and $(n+1,1)$ entries soft-thresholded at level $\epsilon$ when $b$ is large enough. This proves (4). 
 
Consequently, a matrix $A_{n+1}$ satisfying properties (1) to (4) can be obtained by choosing a value of $b$ large enough. This completes the induction. Therefore, for every $n \geq 3$, there exists a matrix $A_n \in \mathbb{P}_{C_n}^+$ such that $\eta_\epsilon(A_n)$ is not positive definite for $\epsilon = \epsilon_0 := 0.1$. 

Now let $\epsilon > 0$ be arbitrary. Notice that for $\alpha > 0$ and any matrix $A$, it holds that 
\begin{equation}
\eta_{\alpha \epsilon}(\alpha A) = \alpha \eta_\epsilon(A).  
\end{equation}
As a consequence, for a given value of $n$, consider the matrix 
\begin{equation}
A := \frac{\epsilon}{\epsilon_0} A_n. 
\end{equation}
Then $A \in \mathbb{P}_{C_n}^+$ since $A_n \in \mathbb{P}_{C_n}^+$. Moreover, 
\begin{equation}
\eta_\epsilon(A) = \eta_{\frac{\epsilon}{\epsilon_0} \epsilon_0}(A) = \frac{\epsilon}{\epsilon_0} \eta_{\epsilon_0}(A_n). 
\end{equation}
Since $\eta_{\epsilon_0}(A_n)$ is not positive definite by construction, it follows that $\eta_\epsilon(A)$ is not positive definite either. This provides the desired example of a matrix $A \in \mathbb{P}_{C_n}^+$ such that $\eta_\epsilon(A)$ is not positive definite.  Therefore, if every matrix $A \in \mathbb{P}_G^+$ retains positive definiteness when soft-thresholded at a given level $\epsilon > 0$, the graph $G$ must not contain any cycle and so is a tree. 

 ($3 \Rightarrow 2$) The implication in this direction holds for more general functions than the soft-thresholding function. The proof is therefore postponed to Section \ref{sec:gen_thres} (see Theorem \ref{thm:pd_G_tree}). 

Finally, since $2 \Rightarrow 1$ trivially, the three statements of the theorem are equivalent. This completes the proof of the theorem. 
\end{proof}

\begin{corollary}[Complete graph case]
For every $n \geq 3$, and every $\epsilon > 0$ there exists a matrix $A \in \mathbb{P}_n^+$ such that $\eta_\epsilon(A) \not\in \mathbb{P}_n^+$. 
\end{corollary}

%%%%%
\section{General thresholding and entrywise maps}\label{sec:gen_thres}
The result of the previous section shows that the commonly used soft-thresholding procedure does not map the cone of positive definite matrices into itself. A natural question to ask therefore is whether other mappings are better adept at preserving positive definiteness. 

In this section, we completely characterize the functions that do so when applied to every off-diagonal element of a positive definite matrix. We begin by introducing some notation and reviewing previous results from the literature for the case where the function is also applied to the diagonal. 
\begin{definition}
For a function $f: \mathbb{R} \rightarrow \mathbb{R}$, denote by:  
\begin{itemize}
\item $f[A]$ the matrix obtained by applying $f$ to every element of the matrix $A$, i.e., 
\begin{equation}
(f[A])_{ij} = f(a_{ij}); 
\end{equation}
\item $f^*[A]$ the matrix obtained by applying $f$ to every element of the matrix $A$, except the diagonal, i.e., 
\begin{equation}
(f^*[A])_{ij} = \left\{\begin{array}{ll}f(a_{ij}) & \textrm{ if } i \not= j \\ a_{ii} & \textrm{  if } i=j.\end{array}\right.
\end{equation}
\end{itemize}
\end{definition}

We now compare $f[A]$ and $f^*[A]$ for $A > 0$. Clearly, $f^*[A] = f[A] + D_A$, where $D_A$ is the diagonal matrix 
\begin{equation}
D_A = \diag(a_{11}-f(a_{11}), \dots, a_{nn}-f(a_{nn})).
\end{equation}
As a consequence, if $f[A] > 0$ and the elements of $D_A$ are nonnegative, then $f^*[A] > 0$. Such is the case when $|f(x)| \leq |x|$. 

\begin{remark}\label{rem:f_and_f_star}
The condition that $|f(x)| \leq |x|$ is a mild restriction which allows us to conclude that $f[A] > 0 \Rightarrow f^*[A] > 0$. As we shall see below, the converse is generally false for matrices of a given dimension. Hence the previous results in the literature characterizing functions which preserve positive definiteness, when the function is also applied to diagonal elements, are unnecessarily too restrictive. In this sense, previous results 	in the field are not directly applicable to problems that arise in modern-day applications. 
\end{remark}

\subsection{Background material: Results for $f[A]$}
It is well-known that functions preserving positive definiteness when applied to every element of the matrix must have a certain degree of smoothness and non-negative derivatives. As we will see later, this is not true anymore when the diagonal is left untouched. 

\begin{theorem}[see Horn {\cite[Theorem 1.2]{Horn_infinitely}}]\label{thm:hj}
Let $f$ be a continuous real-valued function on $(0,\infty)$ and suppose that $f[A]$ is positive semidefinite for every $n \times n$ symmetric positive semidefinite matrix $A = (a_{ij}) $ with positive entries. Then $f$ is $(n-3)$-times continuously differentiable and $f^{(k)}(t) \geq 0$ for every $t \in (0,\infty)$ and every $k=0,\dots,n-3$. 
\end{theorem}

\begin{corollary}
The soft-thresholding operation is not guaranteed to preserve positive semidefiniteness when the diagonal is also thresholded. 
\end{corollary}
\begin{proof}
This follows easily from the non-differentiability of the soft-thresholding function. 
\end{proof}

\begin{corollary}\label{cor:abs_mon_0_inf}
Let $f$ be a continuous function and assume $f[A]$ is positive semidefinite for every symmetric positive semidefinite matrix $A$ with positive entries. Then $f \in C^\infty(0,\infty)$ and $f^{(k)}(t) \geq 0$ for every $t \in (0,\infty)$ and every $k \geq 0$.  
\end{corollary}

Corollary \ref{cor:abs_mon_0_inf} provides a necessary condition for a function $f$ to preserve positive definiteness when applied elementwise to a positive definite matrix. We shall show below that this condition is also sufficient. We first recall some facts about absolutely monotonic functions and the Hadamard product. 

\begin{definition}
Let $0 < \alpha \leq \infty$. A function $f \in C^\infty(0,\alpha)$ is said to be \emph{absolutely monotonic} on $(0, \alpha)$ if $f^{(k)}(x) \geq 0$ for every $x \in (0,\alpha)$ and every $k \geq 0$. 
\end{definition}

The following theorem characterizes the class of absolutely monotonic functions on $(0,\alpha)$. 

\begin{theorem}[see {\cite[Chapter
IV]{lorentz_bernstein_poly}}]\label{thm:abs_monotonic_equiv}
Let $0 < \alpha \leq \infty$. Then the following are equivalent: 
\begin{enumerate}
\item $f$ is absolutely monotonic on $(0,\alpha)$; 
\item $f$ is the restriction to $(0,\alpha)$ of an analytic function on $D(0,\alpha) := \{z \in \mathbb{C} : |z| < \alpha\}$ with positive Taylor coefficients, i.e.,  
\begin{equation}
f(x) = \sum_{n=0}^\infty a_n x^n \qquad (x \in (0,\alpha))
\end{equation}
for some $a_n \geq 0$. 
\end{enumerate}
\end{theorem}

\begin{remark}\label{rem:abs_mon_taylor}
Let $0 < \alpha \leq \infty$. A function $f: (-\alpha, \alpha) \rightarrow \mathbb{R}$ can be represented as: 
\begin{equation}
f(x) = \sum_{n=0}^\infty a_n x^n \qquad (-\alpha < x < \alpha)
\end{equation} 
for some $a_n \geq 0$ if and only if $f$ extends analytically to $D(0,\alpha)$ and is absolutely monotonic on $(0,\alpha)$. 
\end{remark}

Recall that the \emph{Hadamard product} (or Schur product) of two $n \times n$ matrices $A$ and $B$, denoted by $A \circ B$, is the matrix obtained by multiplying the two matrices entrywise, i.e., $(A \circ B)_{ij} = (a_{ij} b_{ij})$. Since $A \circ B$ is a principal submatrix of the Kronecker product $A \otimes B$, the matrix $A \circ B$ is positive definite if both $A$ and $B$ are symmetric positive definite. This last result is commonly known as the \emph{Schur product theorem}.  We now state the converse of Corollary \ref{cor:abs_mon_0_inf}. The proof follows immediately from Theorem \ref{thm:abs_monotonic_equiv} and the Schur product theorem.

\begin{lemma}\label{lem:abs_mon_implies_pos}
Let $0 < \alpha \leq \infty$ and let $f : (0, \alpha) \rightarrow \mathbb{R}$ be absolutely monotonic on $(0,\alpha)$. Then $f[A]$ is positive semidefinite for every symmetric positive semidefinite matrix $A$ with entries in $(0,\alpha)$. 
\end{lemma}

\begin{comment}
\begin{proof}
By theorem \ref{thm:abs_monotonic_equiv}, the function $f$ has a power series expansion:  
\begin{equation}
f(x) = \sum_{n=0}^\infty a_n x^n
\end{equation}
where $a_n \geq 0$. Denote by $f_k$ the partial sums of $f$:
\begin{equation}
f_k(x) = \sum_{i=0}^k a_i x^i, 
\end{equation}
and let $A = (a_{ij})$ be a symmetric positive semidefinite $n \times n$ matrix. Then, by the Schur product theorem, $f_k[A] \geq 0$. Equivalently, for every $\beta \in \mathbb{R}^n$, 
\begin{equation}
\beta^t f_k[A] \beta = \sum_{i,j} f_k(a_{ij}) \beta_i \beta_j \geq 0. 
\end{equation}
Since $f_k(x) \rightarrow f(x)$ for every $x \in (0,\infty)$, it thus follows that $\beta^t f[A] \beta \geq 0$ for every $\beta \in \mathbb{R}^n$, and so $f[A] \geq 0$. 
\end{proof}
\end{comment}

Combining Corollary \ref{cor:abs_mon_0_inf} and Lemma \ref{lem:abs_mon_implies_pos}, and assuming $f$ is continuous, we obtain the following characterization of functions preserving positive definiteness for every positive semidefinite matrix with positive entries. The same result also appears in \cite{vasudeva79}, where it is shown that the continuity assumption is not required. 
\begin{theorem}[{\cite[Theorem 6]{vasudeva79}}]
Let $f: (0,\infty) \rightarrow \mathbb{R}$. Then $f[A]$ is positive semidefinite for every symmetric positive semidefinite matrix $A$ with positive entries if and only if $f$ is absolutely monotonic on $(0,\infty)$. 
\end{theorem}

The following theorem shows that the result remains the same if the entries of the positive semidefinite matrix $A$ are constrained to be in a given interval. Special cases of this result have been proved  by different authors; we state only the most general version here. 
\begin{theorem}[see Schoenberg \cite{Schoenberg42}, Rudin \cite{Rudin59}, Vasudeva \cite{vasudeva79}, Hiai \cite{Hiai2009}, Herz \cite{Herz63}, Christensen and Ressel \cite{Christensen_et_al78}]\label{thm:characterization_diag_thres}
Let $0 < \alpha \leq \infty$ and let $f: (-\alpha, \alpha) \rightarrow \mathbb{R}$. Then $f[A]$ is positive semidefinite for every symmetric positive semidefinite matrix $A$ with entries in $(-\alpha, \alpha)$ if and only if $f$ is analytic on the disc $\{z \in \mathbb{C}: |z| < \alpha\}$ and absolutely monotonic on $(0,\alpha)$. 
\end{theorem}

Recall that one of the primary goals of regularizing positive definite matrices is to ``induce sparsity'', i.e., set small elements to zero. The following result shows that no thresholding function that induces sparsity is guaranteed to preserve positive definiteness. 
\begin{corollary}\label{cor:induce_sparsity}
Let $0 < \alpha \leq \infty$ and let $f: (-\alpha, \alpha) \rightarrow \mathbb{R}$ satisfy $f(0) = 0$ and $f(\gamma) = 0$ for some $\gamma \in (0, \alpha)$. Assume $f \not\equiv 0$ on $(-\alpha, \alpha)$. Then there exists a symmetric positive semidefinite matrix $A$ with entries in $(-\alpha, \alpha)$ such that $f[A]$ is not positive semidefinite. 
\end{corollary}
\begin{proof}
Assume $f[A]$ is positive semidefinite for every symmetric positive semidefinite matrix $A$ with entries in $(-\alpha, \alpha)$. Then, by Theorem \ref{thm:characterization_diag_thres}, 
\begin{equation}
f(z) = \sum_{k=1}^\infty a_k z^k \qquad (z \in D(0,\alpha))
\end{equation}
where $a_k = f^{(k)}(0)/k!\geq 0$. Since $f(\gamma) = 0$, we must have $a_k = 0$ for every $k \geq 1$, i.e., $f \equiv 0$. Thus, if $f \not\equiv 0$, there exists a symmetric positive semidefinite matrix $A$ such that $f[A]$ is not positive semidefinite. 
\end{proof}

\subsection{Preliminary results for $f^*[A]$}

Let $A$ be a symmetric positive definite matrix. We now proceed to analyze mappings $f$ that are applied only to off-diagonal elements of $A$. The next result provides a basic first constraint that $f$ must satisfy in order for $f^*[A]$ to retain positive definiteness. 

\begin{lemma}\label{lem:contraction}
Let $f : \mathbb{R} \rightarrow \mathbb{R}$ and assume $|f(\xi)| > |\xi|$ for some $\xi \in \mathbb{R}$. Then, for every graph $G = (V,E)$ containing at least one edge, there exists a matrix $A \in \mathbb{P}_G^+$ such that $f^*[A]$ is not positive semidefinite. 
\end{lemma}
\begin{proof}
Assume first that $|V| = 2$, and without loss of generality assume $(1,2) \in E$. Since $|f(\xi)| > |\xi|$, there exists $\epsilon > 0$ such that $|f(\xi)| = |\xi| + \epsilon$. Now consider the matrix
\begin{equation}
B := \left(\begin{array}{cc} |\xi|+\frac{\epsilon}{2} & \xi \\ \xi & |\xi| + \frac{\epsilon}{2}\end{array}\right)
\end{equation}
The matrix $B$ is positive definite, but $f^*[B]$ is not positive semidefinite. The general case of a graph with $n$ vertices follows by constructing the matrix $A = B \oplus I_{n-2}$, where $I_k$ denotes the $k \times k$ identity matrix. 
\end{proof}

Recall from Theorem 4.2 that functions preserving positive definiteness when applied to every element of a matrix (including the diagonal) of a given dimension have to be sufficiently smooth, and have non-negative derivatives on the positive real axis. However, when the diagonal is left untouched, the situation changes quite drastically. More precisely, a far larger class of functions preserves positivity, as the following result shows.

\begin{proposition}\label{prop:degree_contraction}
Let $G = (V,E)$ be a connected undirected graph and denote by $\Delta = \Delta(G)$ the maximum degree of the vertices of $G$. Assume $f: \mathbb{R} \rightarrow \mathbb{R}$ satisfies 
\begin{equation}
|f(x)| \leq c |x| \qquad \forall x \in \mathbb{R}, 
\end{equation}
for some $0 \leq c < \frac{1}{\Delta}$. Then $f^*[A] \in \mathbb{P}_G^+$ for every $A \in \mathbb{P}_G^+$.  
\end{proposition}
\begin{proof}
For every $A \in \mathbb{P}_G^+$, denote by $M_A$ the matrix with entries 
\begin{equation}
(M_A)_{ij} = \left\{\begin{array}{ll}
\frac{f(a_{ij})}{a_{ij}} & \textrm{ if } a_{ij} \not= 0 \textrm{ and } i \not=j \\
1 & \textrm{ if } i=j \\
0 & \textrm{ if } a_{ij} = 0 \textrm{ and } i \not= j
\end{array}\right..
\end{equation}
The matrix $f^*[A]$ can be written as 
\begin{equation}
f^*[A] = A \circ M_A.
\end{equation}
Since $0 \leq c < \frac{1}{\Delta}$, an application of Gershgorin's circle theorem demonstrates that $M_A > 0$. As a consequence, by the Schur product theorem, $A \circ M_A > 0$ and so $f^*[A] > 0$ for every $A \in \mathbb{P}_G^+$. 
\end{proof}

\begin{corollary}[Complete graph case]\label{cor:complete_contrac}
Let $n \geq 2$ and assume $f: \mathbb{R} \rightarrow \mathbb{R}$ satisfies 
\begin{equation}
|f(x)| \leq c |x| \qquad \forall x \in \mathbb{R}, 
\end{equation}
for some $0 \leq c < \frac{1}{n-1}$. Then $f^*[A] > 0$ for every $n \times n$ symmetric positive definite matrix $A$. 
\end{corollary}
The following corollary asserts that when operating on the off-diagonal elements, as compared to all the elements (including the diagonals), there are non-trivial functions ``inducing sparsity'' (i.e., setting elements to zero) that preserve positive definiteness. 
\begin{corollary}\label{cor:converse_pd_G_tree}
Let $G$ be a graph and let $0 \in S \subset \mathbb{R}$. Then there exists a function $f: \mathbb{R} \rightarrow \mathbb{R}$ such that: 
\begin{enumerate}
\item $f(x) = 0$ if and only if $x \in S$; 
\item $f^*[A] > 0$ for every $A \in \mathbb{P}_G^+$.
\end{enumerate}
\end{corollary}

\begin{remark}\label{rem:consequences}
Despite the simplicity of the above proofs (especially in contrast to Theorems \ref{Th_Princ_Soft_Thres}, \ref{thm:pd_G_tree}, and \ref{thm:carac_diag_not_thres} of this paper), Proposition \ref{prop:degree_contraction}, Corollary \ref{cor:complete_contrac} and Corollary \ref{cor:converse_pd_G_tree} have important consequences, namely: 
\begin{enumerate}
\item Contrary to the case where the function is also applied to the diagonal elements of the matrix (see Theorem \ref{thm:hj}), Corollary \ref{cor:complete_contrac} shows that, when the diagonal is left untouched, preserving every $n \times n$ positive semidefinite matrix does not imply any differentiability condition on $f$. Even continuity is not required. We therefore note the stark differences compared with previous results in the area. 
\item Proposition \ref{prop:degree_contraction} shows that preserving positive definiteness is relatively easier for matrices that are already very sparse in term of connectivity, i.e., matrices with bounded vertex degree. 
\item Corollary \ref{cor:complete_contrac} suggests that preserving positive definiteness for non-sparse matrices becomes increasingly difficult as the dimension $n$ gets larger. 
\end{enumerate}
\end{remark}

\subsection{Characterization of functions preserving positive definiteness for trees}
Recall that a class of sparse positive definite matrices that is always guaranteed to retain positive definiteness upon either hard or soft-thresholding is the class of matrices with zeros according to a tree (see Theorems \ref{Th_hard_thresholding_level_ep} and \ref{Th_Princ_Soft_Thres}). A natural question to ask therefore is whether functions other than hard and soft-thresholding can also retain positive definiteness. Recall from Lemma \ref{lem:contraction} that for every non-empty graph $G$, the functions $f$ such that $f^*[A] \in \mathbb{P}_G^+$ for every $A \in \mathbb{P}_G^+$ are necessarily contained in the family
\begin{equation}
\mathscr{C} := \{ f : \mathbb{R} \to \mathbb{R}\ :\ |f(x)| \leq |x|\ \forall x \in \mathbb{R}\}. 
\end{equation}
Note that $\mathscr{C}$ is the class of functions contracting at the origin. This ``shrinkage'' property is often required in practice. 

It is natural to ask if we can characterize the set of graphs $G$ for which the functions mapping $\mathbb{P}_G^+$ into itself constitute all of $\mathscr{C}$. The following theorem answers this question.
 
\begin{theorem}\label{thm:pd_G_tree}
Let $G = (V,E)$ be a graph. Then
\begin{equation}
\{f: \mathbb{R} \rightarrow \mathbb{R} : f^*[A] \in \mathbb{P}_G^+ \textrm{ for every } A \in \mathbb{P}_G^+\} = \mathscr{C}
\end{equation}
 if and only if $G$ is a tree. 
\end{theorem}
\noindent
Thus, the result provides a complete characterization of trees in terms of the maximal family $\mathscr{C}$.

\begin{proof}
($\Leftarrow$) Let $G$ be a tree and assume $|f(x)| \leq |x|$ for all $x$. We will prove that $f^*[A] \in \mathbb{P}_G^+$ for every $A \in \mathbb{P}_G^+$ by induction on $n=|V|$. Consider first the case $n=3$. Then $G$ is equal to the $A_3$ graph with 3 vertices

\begin{center}
\includegraphics[width=3cm]{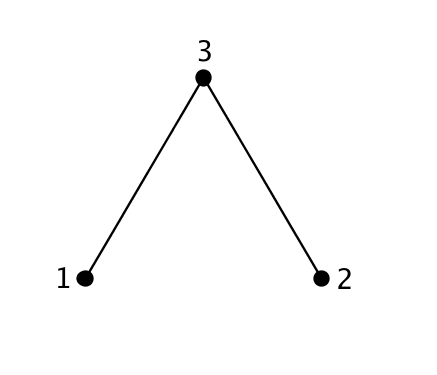}
\end{center}
and $A$ can be reconstituted as follows:  
\begin{equation}
A = \left(\begin{array}{ccc}
\alpha_1 & 0 & a \\
0 & \alpha_2 & b \\
a & b & \alpha_3 
\end{array}\right)
\end{equation} 
with $\alpha_1, \alpha_2, \alpha_3 > 0$. Assume $A > 0$. By computing the determinants of the principal minors, the positive definiteness of $A$ is equivalent to 
\begin{equation}
\alpha_1(\alpha_2\alpha_3-b^2) - \alpha_2 a^2 = \alpha_1 \alpha_2 \alpha_3 - \alpha_1 b^2 - \alpha_2 a^2> 0 . 
\end{equation}
Since $|f(x)| \leq |x|$, it follows that 
\begin{align}
\alpha_1(\alpha_2\alpha_3-f(b)^2) - \alpha_2 f(a)^2 &= \alpha_1\alpha_2\alpha_3 - \alpha_1 f(b)^2 - \alpha_2 f(a)^2 \\
&\geq \alpha_1 \alpha_2 \alpha_3 - \alpha_1 b^2 - \alpha_2 a^2 > 0, 
\end{align}
and so $f^*[A] > 0$. The result is therefore true for $n=3$. 

Assume the result is true for every tree with $n$ vertices and consider a tree $G$ with $n+1$ vertices. Let $\tilde{G}$ be a sub-tree obtained by removing a vertex connected to only one other node. Without loss of generality, assume this vertex is labeled $n+1$ and its neighbor is labeled $n$. Let $A \in \mathbb{P}_G^+$. The matrix $A$ has the form 
\begin{equation}
A = \left(\begin{array}{cccccc}
& & & & & \mc{0} \\
& & & & & \mc{0} \\
& & \widetilde{A} & & & \mc{\vdots} \\
& & & & & \mc{0} \\
& & & & & \mc{a} \\
\cline{1-6}
 0& 0&  \dots &0 & a &  \mc{\alpha}
\end{array}\right).
\end{equation}

By the induction hypothesis, the $n \times n$ principal submatrix $\widetilde{A}$ of $A$ stays positive definite when $f$ is applied to its off-diagonal elements, i.e., $f^*[\widetilde{A}] > 0$. It remains to be shown that the Schur complement of $\alpha$ in $f^*[A]$ is positive definite. Note first that the Schur complement of $\alpha$ in $A$ is given as: 
\begin{equation}
S = \widetilde{A} - \left(\begin{array}{ccc}
0 & \dots & 0 \\
\vdots & \ddots & \vdots \\
0 & \dots & \frac{a^2}{\alpha} 
\end{array}\right). 
\end{equation}
Since by assumption $A > 0$, we have $S > 0$.  We also have $S \in \mathbb{P}_{\widetilde{G}}^+$. Therefore, by the induction hypothesis, $f^*[S] > 0$. Note that $f^*[S]$ is different from the Schur complement of $\alpha$ in $f^*[A]$. More specifically, 
\begin{equation}
f^*[S] = f^*[\widetilde{A}] - \left(\begin{array}{ccc}
0 & \dots & 0 \\
\vdots & \ddots & \vdots \\
0 & \dots & \frac{a^2}{\alpha} 
\end{array}\right), 
\end{equation}
but the Schur complement of $\alpha$ in $f^*[A]$ is given by 
\begin{equation}
f^*[\widetilde{A}] - \left(\begin{array}{ccc}
0 & \dots & 0 \\
\vdots & \ddots & \vdots \\
0 & \dots & \frac{f(a)^2}{\alpha} 
\end{array}\right).
\end{equation}
Since $|f(x)| \leq |x|$ by assumption, it follows that the Schur complement of $\alpha$ in $f^*[A]$ is equal to $f^*[S] + D$ where $D$ is a diagonal matrix with non-negative entries. Since $f^*[S] > 0$, the Schur complement of $\alpha$ in $f^*[A]$ is positive definite, and therefore, so is $f^*[A]$. This completes the induction. 

($\Rightarrow$) Conversely, assume now that $G$ is not a tree and let $\epsilon > 0$. Then, by Theorem \ref{Th_Princ_Soft_Thres}, there exists a matrix $A \in \mathbb{P}_G^+$ such that $(f_\epsilon^S)^*[A] \not\in \mathbb{P}_G^+$, where $f_\epsilon^S$ denotes the soft-thresholding function (see \eqref{eqn:soft_thres}). This concludes the proof. 
\end{proof}

\begin{remark}
A similar result also holds for hard-thresholding with respect to a graph. Indeed, note that every subgraph of a graph $G$ is a union of disconnected induced subgraphs if and only if $G$ is a tree. As a consequence, matrices in $\mathbb{P}_G^+$ are guaranteed to retain positive definiteness when thresholded with respect to any subgraph of $G$ if and only if $G$ is a tree (see Theorem \ref{Th_Hard_Thres_General} and \cite[Corollary 3.5]{Guillot_Rajaratnam2012}). Hence, trees can be characterized by all four types of thresholding operations that have been considered: 1) graph thresholding, 2) hard-thresholding, 3) soft-thresholding, and 4) general thresholding.  
\end{remark}

\begin{remark}
Though Theorem \ref{thm:pd_G_tree} establishes that the class $\mathscr{C}$ is maximal when $G$ is a tree, it is nevertheless important to recognize that even when $G$ is not a tree, there are sparsity inducing functions which retain positive definiteness for all $A \in \mathbb{P}_G^+$ (see Corollary \ref{cor:converse_pd_G_tree})
\end{remark}

\subsection{Proof of the main result}
 We now proceed to completely characterize the functions $f$ preserving positive definiteness for matrices of arbitrary dimension, when the diagonal is not thresholded. 

\begin{theorem}\label{thm:carac_diag_not_thres}
Let $0 < \alpha \leq \infty$ and let $f: (-\alpha, \alpha) \rightarrow \mathbb{R}$. Then $f^*[A]$ is positive semidefinite for every symmetric positive semidefinite matrix $A$ with entries in $(-\alpha, \alpha)$ if and only if $f(x) = x g(x)$ where: 
\begin{enumerate}
\item $g$ is analytic on the disc $D(0,\alpha)$; 
\item $\|g\|_\infty \leq 1$;
\item $g$ is absolutely monotonic on $(0, \alpha)$. 
\end{enumerate}
When $\alpha = \infty$, the only functions satisfying the above conditions are the affine functions $f(x) = ax$ for $0 \leq a \leq 1$.  
\end{theorem}
\begin{proof}
($\Rightarrow$) An important first step in the proof entails showing that if $f^*[A]$ is positive semidefinite for all symmetric positive semidefinite matrices $A$ with entries in $(-\alpha, \alpha)$ (all dimensions), then the same is true for $f[A]$. Let $n \geq 2$ and let $A$ be a symmetric positive semidefinite matrix with entries in $(-\alpha, \alpha)$. Let $D_A$ denote the diagonal matrix 
\begin{equation}
D_A = \diag\left(a_{11}-f(a_{11}), \dots, a_{nn} - f(a_{nn})\right).
\end{equation} 
Note that 
\begin{equation}\label{eqn:f_f_star_rel}
f^*[A] = f[A] + D_A. 
\end{equation}
Denoting by $1_m$ the $m \times m$ matrix with every entry equal to $1$, we obtain for every $m \geq 1$, 
\begin{equation}
f^*[1_m \otimes A] \geq 0. 
\end{equation}
Equivalently, using \eqref{eqn:f_f_star_rel}, 
\begin{equation}\label{eqn:f_1m_plus}
f^*[1_m \otimes A] = f[1_m \otimes A] + I_m \otimes D_A \geq 0, 
\end{equation}
where $I_m$ denotes the $m \times m$ identity matrix. Note that $f[1_m \otimes A] = 1_m \otimes f[A]$. Recall that the eigenvalues of the Kronecker product of two matrices are given by the product of their respective eigenvalues. Therefore, if the eigenvalues of $f[A]$ are given by $\{\lambda_1, \dots, \lambda_n\}$, then the eigenvalues of $f[1_m \otimes A]$ are given by $\{0, m \lambda_1, \dots, m \lambda_n\}$. Applying Weyl's inequality to \eqref{eqn:f_1m_plus}, we obtain 
\begin{equation}
m \lambda_\textrm{min}(f[A]) + \lambda_\textrm{max}(I_m \otimes D_A) \geq \lambda_\textrm{min}(f[1_m \otimes A] + I_m \otimes D_A) \geq 0. 
\end{equation}
Equivalently, 
\begin{equation}
m \lambda_\textrm{min}(f[A]) \geq - \max_{i=1, \dots, n} (a_{ii} - f(a_{ii})). 
\end{equation}
Dividing both sides by $m$ and letting $m \rightarrow \infty$, it follows that $f[A]$ is positive semidefinite for every symmetric positive semidefinite $n \times n$ matrix $A$ with entries in $(-\alpha, \alpha)$.  
Hence, by Theorem \ref{thm:characterization_diag_thres}, $f$ is analytic on $D(0,\alpha)$ and is absolutely monotonic on $(0,\alpha)$, i.e., $f^{(k)}(0) \geq 0$ for every $k \geq 0$. In other words, 
\begin{equation}
f(z) = \sum_{k=1}^\infty a_k z^k
\end{equation}
where $a_k := f^{(k)}(0)/k! \geq 0$. Finally, since $f$ satisfies $|f(x)| \leq |x|$ (see Lemma \ref{lem:contraction}), the function $g$ defined by $g(0) = 0$ and 
\begin{equation}
g(x) := \frac{f(x)}{x} = \sum_{k=1}^\infty a_k x^{k-1} \qquad (x \not= 0), 
\end{equation}
satisfies $|g(x)| \leq 1$ for every $x$, i.e., $\|g\|_\infty \leq 1$. Therefore, $f(x) = x g(x)$ for a function $g$ that is analytic on $D(0,\alpha)$, absolutely monotonic on $(0,\alpha)$, and satisfies the condition $\|g\|_\infty \leq 1$. 

($\Leftarrow$) Conversely, assume $f(x) = x g(x)$ for some function $g$ analytic function on $D(0,\alpha)$,  absolutely monotonic on $(0,\alpha)$, and satisfying $\|g\|_\infty \leq 1$. Then from Theorem \ref{thm:characterization_diag_thres}, $f[A] \geq 0$ for every $A \geq 0$ with entries in $(-\alpha,\alpha)$. Recall that $f^*[A] = f[A] + D_A$ where $D_A$ is the diagonal matrix
\begin{equation}
D_A = \diag(a_{11}-f(a_{11}), \dots, a_{nn} - f(a_{nn})). 
\end{equation}
Since $\|g\|_\infty \leq 1$, then $|f(x)| \leq |x|$ and thus the elements of $D$ are non-negative. Hence, $f^*[A] \geq 0$ for every $A \geq 0$ with entries in $(-\alpha,\alpha)$. 

In the case when $\alpha = \infty$, the only bounded absolutely monotonic functions $g$ on $(0,\infty)$ are the constant functions $g(x) \equiv a$ for some $a \geq 0$. Since $|f(x)| \leq |x|$ we must have $0 \leq a \leq 1$. This completes the proof of the theorem. 
\end{proof}

\begin{remark}
We now compare and contrast Theorem \ref{thm:carac_diag_not_thres} to the analogous results proved in the literature, namely Theorem \ref{thm:characterization_diag_thres}. 
\begin{enumerate}
\item Recall that, for matrices $A \geq 0$ of a given dimension, if $|f(x)| \leq |x|$ and $f[A] \geq 0$, then $f^*[A] \geq 0$. The converse is generally false (see Remarks \ref{rem:f_and_f_star} and \ref{rem:consequences}). 
\item In contrast, for $A \geq 0$ of arbitrary dimension, the above proof shows that $f[A] \geq 0$ and $|f(x)| \leq |x|$ if and only if $f^*[A] \geq 0$. 
\end{enumerate}
\end{remark}

Theorem \ref{thm:carac_diag_not_thres} shows that only a very narrow class of functions are guaranteed to preserve positive definiteness for an arbitrary positive definite matrix of any dimension. In practical applications, thresholding is often performed on \emph{normalized} matrices (such as correlation matrices) which have bounded entries. In that case, more functions preserve positive definiteness. However, as in the case where the function is applied to the diagonal, the following result shows that no thresholding function can induce sparsity (i.e., set non-zero elements to zero) and, at the same time, be guaranteed to maintain positive definiteness for matrices of every dimension.

\begin{corollary}\label{cor:induce_sparsity_diag}
Let $0 < \alpha \leq \infty$ and let $f: (-\alpha, \alpha) \rightarrow \mathbb{R}$ satisfy $f(0) = 0$ and $f(\gamma) = 0$ for some $\gamma \in (0, \alpha)$. Assume $f \not\equiv 0$ on $(-\alpha, \alpha)$. Then there exists a symmetric positive semidefinite matrix $A$ with entries in $(-\alpha, \alpha)$ such that $f^*[A]$ is not positive semidefinite. 
\end{corollary}
\begin{proof}
The proof is the same as the proof of Corollary \ref{cor:induce_sparsity}. 
\end{proof}

%%%%%
\section{Eigenvalue inequalities}\label{sec:eig_inequalities}

The results of Section \ref{sec:gen_thres} show that only a restricted class of functions are guaranteed to preserve positive definiteness when applied elementwise to matrices of arbitrary dimension. Moreover, no function can at the same time induce sparsity (have zeros other than at the origin) and simultaneously preserve positive definiteness for every matrix. Hence, a natural question to ask is whether certain properties of matrices (such as a lower bound on the minimum eigenvalue or an upper bound on the condition number) are sufficient to maintain positive definiteness when a given function $f$ is applied to the off-diagonal elements of the matrix. We provide such sufficient conditions in this section. The results are first derived in Section \ref{subsec:pol_results} for the case when $f$ is a polynomial. They are then extended to more general functions in the subsequent subsection.

\subsection{Bounds for polynomials}\label{subsec:pol_results}

We first establish some notation. For a polynomial $p(x) = \sum_{i=0}^d a_i x^i$, define its ``positive'' and ``negative'' parts by: 
\begin{equation}
p_+(x) = \sum_{a_i > 0} a_i x^i, \qquad p_-(x) = -\sum_{a_i < 0} a_i x^i.
\end{equation}

Many of the results in this section are motivated by the following idea. Note that 
\begin{equation}
p^*[A] = p_+[A] - p_-[A] + D_A
\end{equation}
where $D$ is the diagonal matrix $D_A = \diag(a_{11} - p(a_{11}), \dots, a_{nn} - p(a_{nn}))$. Repeated applications of the Schur product theorem can be used to show that both $p_+[A]$ and $p_-[A]$ are positive definite when $A$ is symmetric positive definite. Intuitively, a polynomial with a positive part that is ``larger'' than its negative part should be able to preserve positive definiteness for a wider class of matrices as compared to a polynomial with a ``large'' negative part. This idea is formalized in Proposition \ref{prop:bound} below. Before stating the result, recall the following classical result that can be used to bound the eigenvalues of Schur products. 

\begin{theorem}[Schur \cite{Schur1911}]
Let $A, B \in \mathbb{P}_n^+$. Then for $i=1, \dots, n$,  
\begin{equation}
\lambda_{\textrm{min}}(A) \min_{k=1, \dots, n} b_{kk} \leq \lambda_i(A \circ B) \leq \lambda_{\textrm{max}}(A) \max_{k=1, \dots, n} b_{kk}. 
\end{equation}
\end{theorem}

\begin{corollary}\label{cor:schur}
Let $A \in \mathbb{P}_n^+$. Denote by $d_{\textrm{min}}$ and $d_\textrm{max}$ the minimal and maximal diagonal element of $A$ respectively. Then for $i=1,\dots, n$, 
\begin{equation}
d_\textrm{min}^{k-1} \lambda_{\textrm{min}}(A) \leq \lambda_i(A^{\circ k}) \leq d_\textrm{max}^{k-1} \lambda_{\textrm{max}}(A), 
\end{equation}
where $A^{\circ k}$ denotes the Hadamard product of $k$ copies of $A$. 
\end{corollary}

We now proceed to state the main result of this subsection. 
\begin{proposition}\label{prop:bound}
Let $p$ be a polynomial and assume $p(0) = 0$. Let $A \in \mathbb{P}_n^+$, and denote by $d_\textrm{min}$ and $d_\textrm{max}$ the smallest and largest diagonal elements of $A$ respectively. Then: 
\begin{enumerate}
\item $\displaystyle\lambda_{\textrm{min}}(p^*[A]) \geq p_+(\lambda_{\textrm{min}}(A)) - p_-(\lambda_{\textrm{max}}(A)) + \min_{i=1,\dots,n} (a_{ii}-p(a_{ii}))$; 
\item $\displaystyle\lambda_\textrm{min}(p^*[A]) \geq \lambda_{\textrm{min}}(A) \frac{p_+(d_{\textrm{min}})}{d_\textrm{min}} - \lambda_{\textrm{max}}(A) \frac{p_-(d_{\textrm{max}})}{d_\textrm{max}} + \min_{i=1, \dots, n} (a_{ii} - p(a_{ii}))$. 
\end{enumerate}
\end{proposition}
\begin{proof}
By Weyl's inequality, 
\begin{equation}
\lambda_\textrm{min}(p[A]) \geq \sum_{a_i > 0} a_i \lambda_{\textrm{min}}(A^{\circ i}) + \sum_{a_i < 0} a_i \lambda_{\textrm{max}}(A^{\circ i}). 
\end{equation}
The first result follows by using Cauchy's interlacing theorem to observe that $\lambda_{\textrm{min}}(A^{\circ i}) \geq \lambda_{\textrm{min}}(A)^i$, $\lambda_{\textrm{max}}(A^{\circ i}) \leq \lambda_{\textrm{max}}(A)^i$, in addition to the fact that $p^*[A] = p[A] + D_A$, where $D_A$ is the diagonal matrix $D_A = \diag(a_{11}-p(a_{11}), \dots, a_{nn} - p(a_{nn}))$.  The second assertion follows by the same argument, but then uses Corollary \ref{cor:schur} to bound the eigenvalues of the Schur product. 
\end{proof}

\begin{corollary}\label{cor:bound1}
Let $n \geq 1$ and let $p$ be a polynomial such that $p(0) = 0$. Let $A \in \mathbb{P}_n^+$ with spectral radius denoted by $\rho(A)$. Then $p^*[A] > 0$ if
\begin{equation}
\min_{i=1,\dots,n} (a_{ii} - p(a_{ii})) \geq p_-(\rho(A)). 
\end{equation}
\end{corollary}

The following surprising result shows that some polynomials having negative coefficients can preserve large classes of positive definite matrices. Recall that a correlation matrix is a symmetric positive definite matrix with ones on the diagonal. 

\begin{corollary}\label{cor:bound2}
Let $n \geq 1$ and let $p$ be a polynomial such that $p(0) = 0$ and $p_-(x) \leq 1-p(1)$ for every $0 \leq x \leq n$. Then $p^*[A] > 0$ for every $n \times n$ correlation matrix $A$.  
\end{corollary}
\begin{proof}
Note that $\lambda_{\textrm{max}}(A) < n$ for every $n \times n$ correlation matrix $A$ since $\tr(A) = n$ and the eigenvalues of $A$ are all positive. The result follows by Corollary \ref{cor:bound1}. 
\end{proof}

Corollary \ref{cor:cond_numb} below shows that $p^* [A]$ is guaranteed to be positive definite if the condition number of $A$ is sufficiently small. Note that the bound becomes more restrictive as the ``negative part'' of $p$ becomes larger compared to its ``positive part''. 

\begin{corollary}\label{cor:cond_numb}
Let $p$ be a polynomial and assume $|p(x)| \leq |x|$ for every $x \in [-a,a]$, for some $a > 0$. Then $p^*[A] > 0$ for every symmetric positive definite matrix $A$ with entries in $[-a,a]$ such that 
\begin{equation}
\cond(A) := \frac{\lambda_{\textrm{max}}(A)}{\lambda_{\textrm{min}}(A)} \leq \frac{p_+(d_\textrm{min})}{p_-(d_\textrm{max})} \frac{d_\textrm{max}}{d_\textrm{min}}. 
\end{equation}
\end{corollary}

\begin{corollary}\label{cor:cond_numb_corr}
Let $p$ be a polynomial and assume $|p(x)| \leq |x|$ for every $x \in [-1,1]$. Then $p^*[A] > 0$ for every correlation matrix $A$ such that 
\begin{equation}
\cond(A) := \frac{\lambda_{\textrm{max}}(A)}{\lambda_{\textrm{min}}(A)} \leq \frac{p_+(1)}{p_-(1)}. 
\end{equation}
\end{corollary}

\subsection{Extension to more general functions}
We now proceed to extend the results of Section \ref{subsec:pol_results} to more general thresholding functions. We first recall the following well-known result. 
\begin{lemma}\label{lemma:approx}
Let $\mathcal{P}^+$ be the set of polynomials with positive coefficients and let $r > 0$. Then the uniform closure of $\mathcal{P}^+$ over $[-r,r]$ is the restriction to $[-r,r]$ of the set of analytic functions $f(z) = \sum_{n \geq 0} a_n z^n$ on the disc $D(0,r) = \{z \in \mathbb{C} : |z| < r\}$ with $a_n \geq 0$ for every $n \geq 0$ and $\sum_{n \geq 0} a_n r^n< \infty$. 
\end{lemma}

\begin{definition}
For $r > 0$ we define
\begin{equation}
W^+(r) = \left\{f(z) = \sum_{n \geq 0} a_n z^n \in \hol(D(0,r)) : \sum_{n \geq 0} |a_n| r^n < \infty\right\}. 
\end{equation}
The space $W^+ := W^+(1)$ is often known as the \emph{analytic Wiener algebra} of analytic functions. The space $W^+(r)$ can be seen as a weighted version of the analytic Wiener algebra. 
\end{definition}
As mentioned above, every function in $W^+(r)$ has a continuous extension to the closed disc $\overline{D}(0,r)$. Notice that $W^+(r) \subset H^\infty(D(0,r))$, the space of bounded analytic functions on the disc $D(0,r)$.  
\ \\

We first begin by extending Proposition \ref{prop:bound}. 

\begin{proposition}\label{prop:bound_ana}
Let $f \in W^+(a)$ for some $a > 0$ and assume $f(0) = 0$. Write $f = f_+ - f_-$ where $f_+, f_- \in W^+(a)$ have nonnegative Taylor coefficients. Then for every $A = (a_{ij}) \in \mathbb{P}_n^+$ with $a_{ij} \in [-a,a]$, we have:   
\begin{enumerate}
\item $\displaystyle\lambda_{\textrm{min}}(f^*[A]) \geq f_+(\lambda_{\textrm{min}}(A)) - f_-(\lambda_{\textrm{max}}(A)) + \min_{i=1,\dots,n} (a_{ii}-f(a_{ii}))$; 
\item $\displaystyle\lambda_\textrm{min}(f^*[A]) \geq \lambda_{\textrm{min}}(A) \frac{f_+(d_{\textrm{min}})}{d_\textrm{min}} - \lambda_{\textrm{max}}(A) \frac{f_-(d_{\textrm{max}})}{f_\textrm{max}} + \min_{i=1, \dots, n} (a_{ii} - f(a_{ii}))$. 
\end{enumerate}
\end{proposition}
\begin{proof}
By Lemma \ref{lemma:approx}, there exist sequences of polynomials with positive coefficients $p_+^{(k)}$, $p_-^{(k)}$ such that 
\begin{equation}
p_+^{(k)} \rightarrow f_+ \textrm{ and } p_-^{(k)} \rightarrow f_-
\end{equation}
uniformly on $[-a,a]$ as $k \rightarrow \infty$. Let $p^{(k)} := p_+^{(k)} - p_-^{(k)}$. An application of the triangle inequality shows that $p^{(k)} \rightarrow f$ uniformly on $[-a,a]$. Now, by Proposition \ref{prop:bound}, 
\begin{equation}
\lambda_{\textrm{min}}(p^{(k)}[A]) \geq p_+^{(k)}(\lambda_{\textrm{min}}(A)) - p_-^{(k)}(\lambda_{\textrm{max}}(A)) + \min_{i=1,\dots,n} (a_{ii} - p^{(k)}(a_{ii})). 
\end{equation}
The first result follows by the continuity of the eigenvalues and uniform convergence. The second part follows similarly. 
\end{proof}

Corollary \ref{cor:cond_numb} can also be easily extended using a similar argument. 

\begin{theorem}\label{thm:cond_numb_ana}
Let $f \in W^+(a)$ for some $a > 0$ and assume $|f(x)| \leq |x|$ for every $x \in [-a,a]$. Then $f^*[A] > 0$ for every symmetric positive definite matrix with entries in $[-a,a]$ such that 
\begin{equation}
\cond(A) \leq \frac{f_+(d_\textrm{min})}{f_-(d_\textrm{max})} \frac{d_\textrm{max}}{d_\textrm{min}}, 
\end{equation}
where $d_{\textrm{min}}$ and $d_\textrm{max}$ denote the minimal and maximal diagonal element of $A$ respectively. 
\end{theorem}

The previous results can easily be extended to more general functions that can be approximated pointwise by polynomials. The following result illustrates this idea for Corollary \ref{cor:cond_numb}.  
\begin{theorem}
Let $f : [-a,a] \rightarrow \mathbb{R}$ for some $a > 0$ and assume $|f(x)| \leq |x|$ for every $x \in [-a,a]$. Moreover, assume $f$ is the pointwise limit of a sequence of polynomials $p^{(n)}$ on $[-a,a]$. Then $f^*[A] > 0$ for every symmetric positive definite matrix with entries in $[-a,a]$ such that
\begin{equation}
\cond A \leq \limsup_{n \rightarrow \infty}  \frac{p^{(n)}_+(d_\textrm{min})}{p^{(n)}_-(d_\textrm{max})} \frac{d_\textrm{max}}{d_\textrm{min}}, 
\end{equation}
where $d_{\textrm{min}}$ and $d_\textrm{max}$ denote the minimal and maximal diagonal element of $A$ respectively. 
\end{theorem}
\begin{proof}
Since $p^{(n)}[A]$ converges to $f[A]$ entrywise as $n \rightarrow \infty$, the eigenvalues of $p^{(n)}[A]$ converge to the eigenvalues of $f[A]$. The result follows from Corollary \ref{cor:cond_numb} using a limiting argument similar to the one in the proof of Proposition \ref{prop:bound_ana}. 
\end{proof}

\begin{acknowledgement}
We wish to thank Apoorva Khare for valuable comments and discussions. We also acknowledge an anonymous referee for useful comments. Dominique Guillot acknowledges support from NSERC (Canada) and National Science Foundation Grants DMS-CMG-1025465 and DMS-1106642. Bala Rajaratnam was supported in part by the National Science Foundation under Grant Nos. DMS-0906392 (ARRA), DMS-CMG-1025465, AGS-1003823, DMS-1106642 and grants NSA H98230-11-1-0194, DARPA-YFA N66001-11-1-4131, AFOSR FA9550-13-1-0043 and SUWIEVP10-SUFSC10-SMSCVISG0906.
\end{acknowledgement}

\bibliographystyle{plain}
\bibliography{biblio}

\end{document}